\documentclass[11pt]{amsart}
\usepackage{amsmath,amssymb,amsthm, mathtools}
\usepackage{amsaddr}
\usepackage{graphics}
\usepackage{graphicx}
\usepackage{epstopdf}
\usepackage{hyperref}
\usepackage{epsfig,color}
\usepackage{psfrag}

\newtheorem{conjecture}{Conjecture}
\newtheorem{theorem}{Theorem}
\newtheorem{lemma}{Lemma}
\newtheorem{proposition}{Proposition}
\newtheorem{problem}{Problem}
\newtheorem{corollary}{Corollary}

\newcommand \B {{\mathcal B}}

\newcommand{\inc}{\!\parallel\!}

\newcommand \Po {{\mathcal P}}
\newcommand{\dw}{\ \! \downarrow\!\!}

\newcommand{\Bn}{\mathcal{B}(n)}

\newcommand{\C}{\mathcal{C}}
\newcommand{\co}{\mathrm{comp}}
\newcommand{\D}{\mathcal{D}}

\newcommand{\2}{\mathbf{2}^n}

\begin{document}

\parindent = 0cm
\parskip = .25cm

\vspace*{-2cm}

\title[Covering distributive lattices by intervals]{Covering distributive lattices by intervals}\vspace{-.5cm}

\vspace{-.5cm}
\author{Dwight Duffus\vspace{-.5cm}}
\address {Mathematics Department\\
  Emory University, Atlanta, GA  30322 USA\vspace{-.5cm}}
   \email{dwightduffus@emory.edu}
\author{Bill Sands\vspace{-.5cm}}
\address{Mathematics and Statistics Department\\
   The University of Calgary, Calgary, AB T2N 1N4 Canada\vspace{.5cm}}
  \email{sands@ucalgary.ca}
  \thanks{{\it Acknowledgement}: The second author thanks Emory University and the University of Victoria for their
  support of visits during which he conducted some of the research for this article.}
  
\date{\today}

\begin{abstract}
We consider the convex subset $[A,B]$ of all elements between two levels $A$ and $B$ of a finite 
distributive lattice, as a union of (or {\it covered by}) intervals $[a,b]$. A 1988 result of Voigt and Wegener 
shows that for such convex subsets of finite Boolean lattices, covers using $\max(|A|,|B|)$ intervals 
(the minimum possible number) exist. In 1992 Bouchemakh and Engel pointed out that this result holds more generally 
for finite products of finite chains. In this paper we show that covers of size $\max(|A|,|B|)$ exist for $[A,B]$ when $A$ 
is the set of atoms and $B$ the set of coatoms of any finite distributive lattice. This is a consequence of a more 
general result for finite partially ordered sets. We also speculate on the situation when other levels of finite 
distributive lattices are considered, and prove a couple of theorems supporting these speculations.
\end{abstract}

\maketitle

\thispagestyle{empty}

\section{Introduction}\label{S:intro}
Let $P$ be a finite partially ordered set and recall that an {\it order ideal} or {\it downset} $D$ 
(respectively, {\it order filter} or {\it upset} $U$) of $P$ satisfies
$x\in D$ whenever $x\le d$ in $P$ for some $d\in D$ (respectively, $y\in U$ 
whenever $y\ge u$ in $P$ for
some $u\in U$). For any $S\subseteq P$, let $\downarrow\!\!S 
= \{x\in P\;|\; x\le s \mbox{ for some } s\in S\}$ denote the order ideal 
generated by $S$ and $\uparrow\!\! S$, the order filter generated by $S$, defined dually. If 
$S=\{s\}$ then we write $\downarrow\!\!s$ instead of $\downarrow\!\!\{s\}$, 
and similarly for $\uparrow\!\!s$. Call $S$ a {\it convex} subset of $P$
if whenever $s\le x\le s'$ in $P$ with $s, s'\in S$, it follows that $x\in S$.  Note that
downsets and upsets are examples of convex subsets.

For subsets $A,B$ of a partially ordered set $P$, write $A\le B$ to mean 
that  for all $a\in A$ there is $b\in B$ so that $a\le b$, and for all 
$b\in B$ there is $a\in A$ so that $a\le b$. This is equivalent to
$A\subseteq\downarrow\!\!B$ and $B\subseteq\uparrow\!\!A$. 

If $A$ and $B$ are 
respectively the minimals and maximals of a convex subset $S$ then it is easy to 
see that $A$ and $B$ are antichains satisfying $A\le B$. Conversely, if $A$ 
and $B$ are antichains of $P$ satisfying $A\le B$, then the subset 
$$[A, B] \ =\ \{x\in P:a\le x\le b\mbox{ for some }a\in A,b\in B\}$$ 
is the smallest convex subset of $P$ containing $A\cup B$, that is, 
the {\it convex hull} of $A\cup B$. 
In the special case $A = \{a\}$ and $B = \{b\}$, $A\le B$ says that $a\le b$, 
and we replace $\left[\{a\}, \{b\}\right]$ by the standard notation $[a, b]$ for the {\it interval} 
determined by $a$ and $b$.

Here is the problem of interest.

\begin{problem}\label{Prob:main}
For antichains $A,B\subseteq P$ satisfying 
$A\le B$, determine the minimum number of intervals
$[a,b]$, where $a\in A$ and $b\in B$, that are required to cover $[A, B]$.
\end{problem}

Obviously,  at least $\max(|A|,|B|)$ intervals are required. Equally obviously,
$|A|\cdot|B|$ intervals will always suffice.

For the lower bound $\max(|A|,|B|)$, let us say that the convex hull $[A,B]$ of levels $A$ and $B$ in a ranked partially ordered set $P$ is 
{\it minimally coverable} if it can be covered by $\max(|A|,|B|)$ intervals.  (For a formal definition of levels, see Section \ref{S:ICP}.)
Voigt and Wegener  \cite{VW} showed that 
if $P= \mathcal{B}(n)$ is a Boolean lattice, then every such $[A,B]$ is minimally coverable. Bouchemakh and Engel   \cite{BE}
extended this result to arbitrary finite products of finite chains.

The result by Voigt and Wegener does not extend to all convex hulls generated by subsets of levels of the Boolean lattice  $\mathcal{B}(n)$.
Far from it:  the following family of examples, due to Peter Frankl, shows that such convex hulls sometimes require the maximum number of intervals to be covered rather than the minimum.  For disjoint nonempty subsets $X = \{x_1, x_2, \ldots , x_r\}$ and $Y = \{y_1, y_2, \ldots, y_s\}$ 
of $[n]$ let
$$\mathcal{X} = \{ \{x_1\}, \{x_2\}, \ldots , \{x_r\}\} \  \text{and} \  \mathcal{Y} = \{ X \cup \{y_1\},  X \cup \{y_2\}, \ldots,  X \cup \{y_s\}\}.$$
Then $\mathcal{X} \le \mathcal{Y}$, so that $[\mathcal{X}, \mathcal{Y}]$ is defined.  Also, for each $i = 1, 2, \ldots, r$ and $j = 1, 2, \ldots, s$, $\{x_i, y_j\} \in [\mathcal{X}, \mathcal{Y}]$ and the only interval defined by a member
of $\mathcal{X}$ and one of $\mathcal{Y}$ that contains $\{x_i, y_j\}$ is $[\{x_i\}, X \cup \{y_j\}]$.  Thus, $r \cdot s$ intervals are required
to cover $[\mathcal{X}, \mathcal{Y}]$.

A similar construction gives a graded partially ordered set, a lattice actually, containing entire levels $A$ and $B$ for which $|A|\cdot|B|$ 
intervals are required to cover $[A, B]$. The lattice  $L$ in Figure \ref{F:doubly} has $A = \{a_1, a_2, \ldots, a_r\}$ as the
atoms, $C =\{c_1, c_2, \ldots, c_s\}$ as the coatoms (replacing $B$), with $a_i < x < c_j$ for all $i, j$. (The elements $0, x, 1$
are required for $L$ to be a lattice.)  Finally, for all $i, j$, there is an element 
$x_{i,j}$ that has $a_i < x_{i, j} < c_j$ and is in a unique atom-coatom interval, $[a_i, c_j ]$.  
Thus, we require $|A|\cdot|C|$  intervals to cover $[A,C]$.

\begin{figure}[h!]
   \scalebox{.32}{\includegraphics{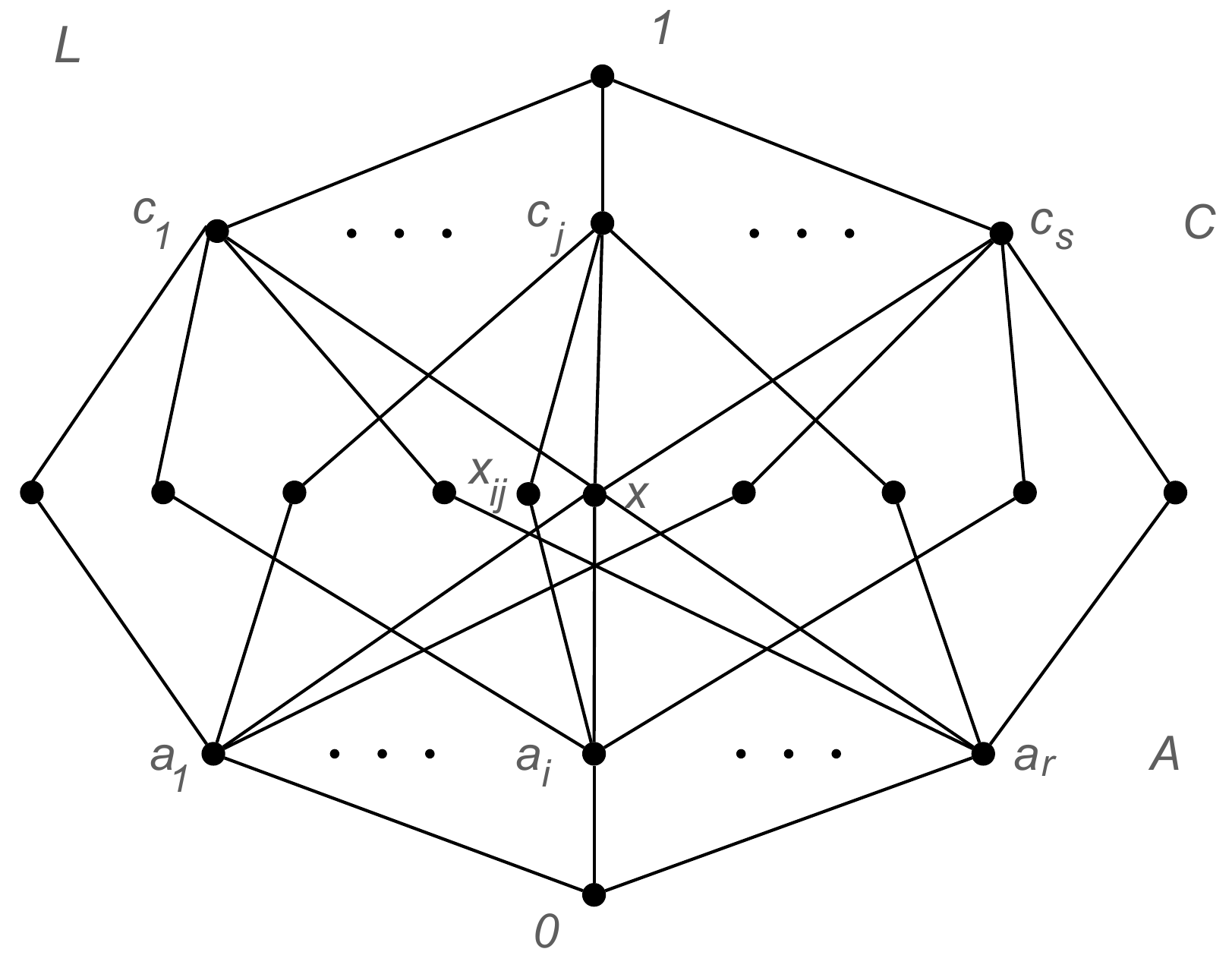}}
   \caption{Covering $[A,C]$ requires $|A| \cdot |C|$ intervals in the lattice $L$}\label{F:doubly}
\end{figure}

As we shall see, not all convex hulls of levels in finite distributive lattices are minimally coverable. But in one important case, they are:

\begin{theorem}\label{T:distributive}
Let ${\mathcal A}$ and ${\mathcal C}$ be the atoms and 
coatoms, respectively, of a finite distributive lattice ${\mathcal D}$, with  $|{\mathcal D}| \ge 3$. Then 
$[{\mathcal A},{\mathcal C}] = {\mathcal D}-\{0,1\}$ is the 
union of $\max(|{\mathcal A}|,|{\mathcal C}|)$ intervals.
\end{theorem}

Theorem \ref{T:distributive} cannot be extended to all pairs of levels in a distributive lattice.
For example, in the lattice $L$ consisting of two copies of the Boolean lattice $\mathcal{B}(3)$ overlapping 
in an edge, the levels $L_2$ and $L_3$ both 
have three elements, but any cover of $[L_2,L_3] = L_2 \cup L_3$ requires four 
intervals. This is easily generalized to produce a distributive lattice 
with consecutive levels of sizes $n$ and $m$ whose union
requires $m+n-2$ intervals -- see Figure \ref{F:Glued}
for diagrams of each. We will 
have more to say about the distributive lattice situation in Section \ref{S:more-dist}.

\begin{figure}[h!]
   \scalebox{.38}{\includegraphics{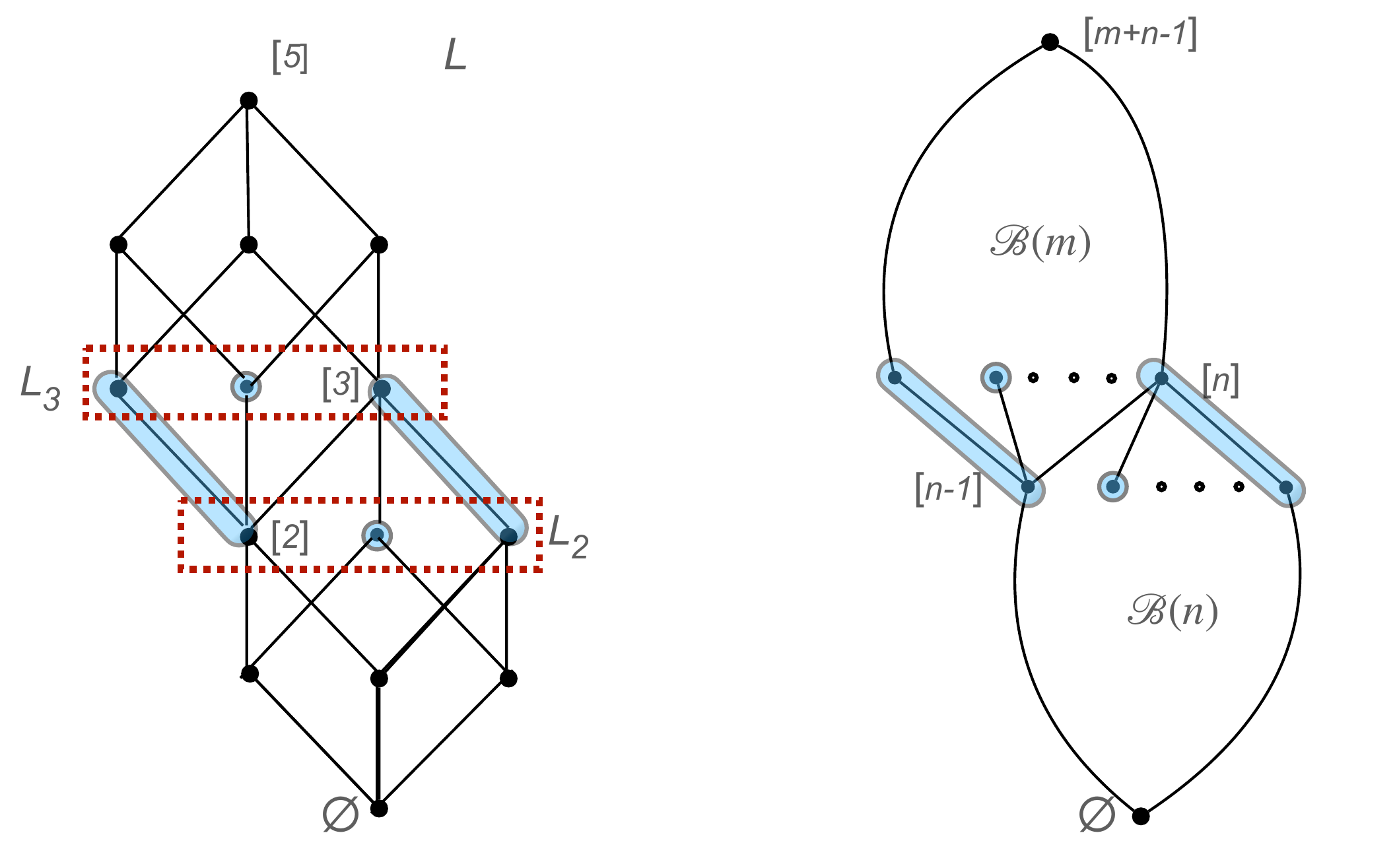}}
   \caption{Covering $L_2 \cup L_3$ requires 4 intervals. On the right, covering levels $n-1$ and $n$, of 
   sizes $n$ and $m$ respectively, requires $m+n-2$ intervals.}\label{F:Glued}
\end{figure}

Theorem \ref{T:distributive} will be a consequence of the following stronger result on partially 
ordered sets, which may be of independent interest.\\[-.2cm]

\begin{theorem}\label{T:poset}
 Let $A,B$ be antichains of a partially ordered set $P$ with 
$|A|\le|B|$ and so that $A\le B$ in $P$. Then there exists an onto 
function $f:B\to A$ so that, for any order ideal $X$ of 
$P$ satisfying $A\cap X\ne\emptyset$ and $B\not\subseteq X$, 
there is some $b\in B-X$ with $f(b)\in X$.
\end{theorem}

If one wishes to remove the restriction in Theorem 1 that the levels in question be the atoms and coatoms, 
one obvious case in which the theorem still holds is when one of the levels in question has size 1. Also, 
the first example in Figure 2 shows that ``size 1" cannot be replaced by``size 3". So what about size 2? 
Our next result handles this case to happy effect.

  \begin{theorem}\label{T:size2}
  Let $D$ be a distributive lattice with distinct levels $A$ and $B$ that satisfy $A \le B$ and $\min(|A|, |B|) = 2$.  
  Then $[A, B]$ has a cover by $\max(|A|, |B|)$ intervals.
  \end{theorem}
  
  Finally, we make a modest improvement on  the obvious upper bound on the number of intervals required to cover
  the convex subset of a distributive lattice generated by two of its levels.

  \begin{theorem}\label{T:slight}
  Let $D$ be a distributive lattice with distinct levels $A$ and $B$ that satisfy $A \le B$ and $|A|, |B| > 1$. 
  Then $[A, B]$ has a cover by at most

\begin{enumerate}

  \item $|A| \cdot |B| - \min (|A|, |B|)$ intervals in $D$, if $|A| \ne |B|$ or $|A| = |B|$ is even,\\
  
  \item $|A| \cdot |B| -  \min (|A|, |B|)+ 1$ intervals in $D$, if  $|A| = |B|$ is odd.
  
\end{enumerate}

\end{theorem}

Here is how the paper is organized. In the next section, we discuss the genesis of this paper, and expand on some aspects of 
the Voigt-Wegener \cite{VW} result that $\Bn$ is minimally coverable  (see the strong interval cover property defined in the
next section).
Necessary preliminaries  and the similarity of Theorem \ref{T:poset} to a distributive lattice result
are presented in Section \ref{S:distributive}. In Section \ref{S:Proof}, we prove Theorem \ref{T:poset}
and derive Theorem \ref{T:distributive}. In Section \ref{S:more-dist} we give proofs of Theorems \ref{T:size2} and \ref{T:slight}
along with our speculations on how these two theorems might be improved.

\section{The Interval Cover Property and Covers of The Boolean Lattice $\Bn$}\label{S:ICP}

Our original motivation for studying covers of convex subsets by intervals was the following 40 year-old
conjecture.  The {\it width} of a partially ordered set $P$, denoted by $w(P)$, is the maximum size of an antichain in $P$.

\begin{conjecture} {\rm(Daykin and Frankl \label{C:convex}\cite{DF})}
For any nonempty convex subset $\C$ of  $\Bn$,\\
$$\frac{w(\C)}{|\C|}\ \ge \ \frac{\binom{n}{\lfloor n/2 \rfloor}}{2^n} \ .$$
\end{conjecture}

According to Sperner's theorem, the width of $\Bn$ is the size of a middle level, so the conjecture states that the ratio
of width to cardinality of convex subsets of $\Bn$ is minimized by the full lattice.

The function $f(k) = \binom{k}{\lfloor k/2\rfloor}/2^k$ is nonincreasing on the positive integers,
so the conjecture holds for those convex $\C$ that are intervals in $\Bn$.  This led us to ask if the ratio of
width to cardinality of a convex set is related to its covering by intervals.  We rediscovered work by Voigt and
Wegener \cite{VW} in the process.

They used coverings of convex subsets of $\Bn$ by intervals to construct minimal polynomials for symmetric 
 Boolean functions.  Their main tool for obtaining coverings is the  Greene-Kleitman 
symmetric chain decomposition of $\Bn$ \cite{GK}.  Later, Bouchemahk and Engel \cite{BE} provided a general
setting for studying interval covers and related covers to symmetric chain decompositions.  We first introduce the
tools from \cite{BE}, sketch the proof of the Voigt-Wegener result (Theorem \ref{T:Boolean} below) provided in \cite{BE} and then show how 
the Greene-Kleitman decomposition provides an explicit way to define interval covers.

We use $\prec$ to denote the {\it covering relation} in a finite partially ordered set: $x \prec y$ if $x < y$ and there is no $z$ such
that $x < z < y$.  In this case, we say that $y$ {\it covers}  $x$ and that $y$ is an {\it upper cover} of $x$; {\it covered by} and {\it lower
cover} have the obvious meanings.
We say that $P$ is  {\it ranked} if there is a function $r$ on $P$
satisfying $r(x) = 0$ for all minimal elements $x$ and $r(z) = r(y) +1$ whenever $y \prec z$ in $P$.  Call $r$ the
{\it rank function} of $P$, let the {\it rank} of $P$, $r(P)$, be the maximum of $r(x)$ for $x \in P$,  let
$$P_i = \{ x \in P : r(x) = i\} \ (i = 0, 1, \ldots, r(P)) \ \ \text{and} \ \ P_{j,k} = \bigcup_{i =j}^k P_i  \ (0 \le j \le k \le r(P)).$$
Call $P_i$ the $i^{\mathrm{th}}$ {\it level} of $P$, $r_i = |P_i|$, the $i^{\mathrm{th}}$ {\it rank number} or {\it level number}.
Also, notice that $P_{j,k} = [P_j, P_k]$ and that $r(P_{j,k}) = k - j$. 

Following \cite{BE}, a {\it covering by intervals of} $P$ is a family of intervals whose union is $P$ and
$\rho(P)$ is the minimum number of intervals in such a family.   Note that $\rho(P_{j,k}) \ge \max \{r_j, r_k\}$ always holds. 
Say that $P$ has the $(j, k)$-{\it interval covering property} (or  the $(j, k)$-{\it ICP}) if
$$\rho(P_{j,k}) = \max \{r_j, r_k\}$$
and that $P$ has the {\it strong interval covering property}  (or the {\it strong ICP}) if $P$ has the $(j,k)$-ICP for all $j, k$ with
$0 \le j \le k \le r(P).$  

Note that with this terminology Theorem \ref{T:distributive} states that every finite distributive lattice of rank $n$ ($n \ge 2$) has 
the $(1, n-1)$-ICP.

 Recall that a chain $C = \{c_0, c_1, \ldots, c_m\}$ in a ranked partially ordered set $P$ is called 
a {\it symmetric chain} provided $c_i \prec c_{i+1}$ ($i = 0, 1, \ldots, m-1$) and $r(c_0) + r(c_m) = r(P)$.
A family of symmetric chains $\mathcal{C}(P)$ that partitions $P$ is a {\it symmetric chain decomposition (SCD)} of
$P$ and $P$ is a {\it symmetric chain order (SCO)}.   If $P$ has an SCD $\mathcal{C}(P)$ such that
for all $C \in \mathcal{C}(P)$ of less than maximum length there exists $D \in \mathcal{C}(P)$ such that
  \begin{equation*}
     (\mathbf{\star})\  \  \min D  \prec \min C \le \max C \prec \max D
  \end{equation*}
then, following \cite{BE}, call $P$ a {\it special symmetric chain order (SSCO)} and $\mathcal{C}(P)$ a {\it special
symmetric chain decomposition (SSCD)}. 

Here is a restatement of the Voigt-Wegener result for $\Bn$ .

\begin{theorem}\label{T:Boolean}(\cite{VW}, cf. \cite{BE})
The Boolean lattice $\Bn$ has the strong interval cover property.
\end{theorem}

Theorem \ref{T:Boolean} can be obtained from two results in \cite{BE}: a ranked partially ordered set $P$ with an SSCD
has the strong ICP (\cite{BE}, Theorem 5); and, if ranked partially ordered sets $P$ and $Q$ are both SSCO's then $P \times Q$
is an SSCO (\cite{BE}, Theorem 6).  Notice that these results show that any product of chains is an SSCO and has the strong ICP 
(\cite{BE}, Example 2).  So, we know that $\Bn$ is an SSCO.  In fact, the Greene-Kleitman SCD of $\Bn$, which we will
denote by $\C^{\star}(\Bn)$, is an SSCD.  This follows directly from the construction of this SCD, as we shall see below.
We note that Griggs, Killian and Savage \cite{GKS} give a detailed proof of this in their investigations of symmetric Venn diagrams.

Regard the subsets $A$ of $[n]$ as binary sequences of length $n$, say $A = a_1 a_2 \cdots a_n$ where $a_i = 1$
if and only if  $i \in A$.  Coordinates or positions are categorized as follows.  Scan from left to right.  If positions 1 through
$i-1$ are scanned and $a_i = 0$ then $i$ is unpaired (for now).  If $a_i = 1$ then pair $i$ with the rightmost position, say $j$,
such that $a_j = 0$ is unpaired, if such $j$ exists.  Now $i$ and $j$ are paired.  If no such $j$ exists, $i$ is unpaired.
In the end, we have the unpaired positions containing 0, say $U_0(A)$, the unpaired positions containing 1, say $U_1(A)$, and
the set $P(A) = \{(j, i) : j < i,  a_j = 0, a_i = 1 \ \text{and} \ i, j \ \text{are paired}\}.$  The symmetric chain containing $A$
is obtained by: (1) flipping the 0's indexed by $U_0(A)$ to 1's (from left to right), giving the members of the chain larger than $A$; 
and, (2) flipping the 1's indexed by  $U_1(A)$ to 0's (from right to left), filling out the chain below $A$.  The chain decomposition
obtained has these properties:

\begin{enumerate}

  \item for all $B \in \Bn$, $B$ is in the chain containing $A$ if and only if $P(B) = P(A)$;\\
  
  \item the minimum element $X$ of the chain has $U_1(X) = \emptyset$ and the maximum
  element $Y$ has $U_0(Y) = \emptyset$; and,\\
  
  \item the length of the chain is $U_0 (A) + U_1 (A)$.

\end{enumerate}
Let $\C =  \C^{\star}(\Bn)$ for brevity.
We now define a function $\phi$ that maps each $C \in \C$, of length less than $n$, to some $D \in \C$
such that $(\mathbf{\star})$ holds.  Let $\min C = X$ and $\max C = Y$.  By {\bf(1)} and {\bf(3)}, $P(X) = P(Y) \ne \emptyset$
so we choose the $(j, i) \in P(X)$ such that $i$ is maximum among all pairs in $P(X)$.  Let $X' = X -\{i\}$, $Y' = Y \cup \{j\}$.
By the choice of $(j, i)$ as the rightmost paired coordinates of $X$ and $Y$,  $P(X') = P(X) -\{(j, i)\} = P(Y) -\{(j, i)\} = P(Y')$. 
By {\bf (2)}, $X'$ is the minimum element of its chain in $\C$,
$Y'$ is the maximum element of its chain in $\C$, and, by {\bf (1)}, $X'$ and $Y'$ are in the same chain, say $D$.
Let $\phi(C) = D$ and note that $C$ and $D$ satisfy $(\mathbf{\star})$.

Let $0 \le j \le k \le n$ and $\B = \Bn_{j, k}$ be the convex subset of $\Bn$ of all elements between levels $j$ and $k$.  We
may assume without loss of generality that $j + k \le n$, so $\binom{n}{j} \le \binom{n}{k}$.  We need to define a function $\psi$
of level $k$ of $\Bn$ to level $j$ such that $\B = \cup_{|Y| = k} [\psi(Y), Y]$.   

For $Y \in \Bn$ of cardinality $k$, 
let $Y \in C \in \C$.  If $C$ contains an element $X$ of cardinality $j$ then let $\psi(Y) = X$.  Otherwise, $|\min C| > j$, say
$|\min C| = j + m$.  Since $|\min \phi(C)| = |\min C| - 1$, we see that $|\min \phi^m(C)| = j$. In this case, let $\psi(Y) = \min \phi^m(C)$.

Let $Z \in \B$ with $Z \in C \in \C$.  If $|\max C| \ge k$, let $Y \in C$ with $|Y| = k$.  Then $|\psi(Y)| = j \le |Z|$ and either $\psi (Y) \in C$
or $\psi (Y) \subseteq \min C$.  In either case $\psi(Y) \subseteq Z$ so $Z \in [\psi(Y), Y]$.   If $|\max C| < k$ then since $0 \le j \le k$ and
$j + k \le n$, 
$$|\max C| = k - s, |\min C| = j + t, \ \text{with} \ t \ge s.$$
We have that $|\max \phi^s(C)| = k$, $|\min \phi^t (C)| = j$,  and 
$$\min \phi^t(C) \subseteq \min C \subseteq Z \subseteq \max C \subseteq \max \phi^s(C).$$
Since $\phi^t (C) = \phi^{t-s}(\phi^s(C))$, $\psi(\max \phi^s(C)) = \min \phi^t (C)$, so $Z \in [\psi(\max \phi^s(C)),  \max \phi^s(C)]$.

This description of the cover by a minimum possible number of intervals in $\Bn$ can be generalized to any finite product of
finite chains because the pairing procedure for $\Bn$ that gives the SSCD has an analog for arbitrary products of chains (see
\cite{E}, Example 5.1.1).

\section{Convex Subsets of Finite Distributive Lattices}\label{S:distributive}

We begin this section with a review of the Birkhoff duality between finite 
partially ordered sets and finite distributive lattices.  This requires some
terminology and notation.

Let $P$ be a partially ordered set. Let 
${\mathcal D}={\mathcal O}(P)$ be the distributive lattice of order ideals of 
$P$ ordered by containment.  Then $P$ is isomorphic to the set $J({\mathcal D})$ of 
join-irreducibles of ${\mathcal D}$, which corresponds to the set 
$\{\downarrow\!\!x:x\in P\}$ of all principal order ideals of $P$. Note that $P$ is also 
isomorphic to the set $M({\mathcal D})$ of meet-irreducibles of 
${\mathcal D}$, which corresponds to the set $\{P-\uparrow\!\!x:x\in P\}$. 

For $a\in P$, let 
$${\bf a}=\downarrow\!\!a\in J({\mathcal D}),\quad{\bf\overline{a}}
=P-\uparrow\!\!a\in M({\mathcal D}).$$

For a subset $A$ of $P$, we define
$${\mathcal A}=\{{\bf a}\;|\;a\in A\}\subseteq J({\mathcal D}),\quad\overline
{\mathcal A}=\{\overline{\bf a}\;|\;a\in A\}\subseteq M({\mathcal D}).$$
Then ${\mathcal A}\cong A\cong\overline{\mathcal A}$. 

For subsets $A,B$ of a poset $P$, write $A\not\ge B$ to mean that: for all 
$a\in A$ there is $b\in B$ so that $a\not\ge b$, and for all $b\in B$ there is 
$a\in A$ so that $a\not\ge b$. This is equivalent to: 
$B\not\subseteq\downarrow\!\!a$ for any $a\in A$ and 
$A\not\subseteq\uparrow\!\!b$ for any $b\in B$. [{\it Caution}: $A\not\ge B$ 
is not the negation of $B\le A$!].  

\begin{lemma}\label{L:orders2}
Let $A,B$ be subsets of a partially ordered set $P$.  

(i) If $A$ has no minimum, $B$ has no maximum, and $A\le B$, 
then $A\not\ge B$.

(ii) $A\not\ge B$ if and only if ${\mathcal A}\le\overline{\mathcal B}$ in 
${\mathcal D}={\mathcal O}(P)$.
\end{lemma}

\begin{proof}
(i) Suppose that $A$ has no minimum, $B$ has no maximum, and 
$A\le B$. Let $a\in A$; then there is $b\in B$ so that $a\le b$, so 
$a\not\ge b$ unless $a=b$. But if $a=b$, since $b$ is not the maximum of $B$ 
there is $b'\in B$ with $b\not\ge b'$, and thus $a\not\ge b'$. Similarly, let 
$b\in B$; then there is $a\in A$ so that $a\le b$, so $a\not\ge b$ unless 
$a=b$. But if $a=b$, since $a$ is not the minimum of $A$, there is $a'\in A$ 
with $a'\not\ge a$, and thus $a'\not\ge b$. Hence $A\not\ge B$.

(ii) Assume that $A\not\ge B$ in $P$. In ${\mathcal D}$, let 
${\bf a}\in{\mathcal A}$, which corresponds to $\downarrow\!\!a$ for an 
element $a\in A$. Then there exists $b\in B$ so that $a\not\ge b$. Thus 
$a\in P-\uparrow\!\!b$, so ${\bf a}\le\overline{\bf b}$ in ${\mathcal D}$, 
where $\overline{\bf b}\in\overline{\mathcal B}$. Similarly, let 
$\overline{\bf b}\in\overline{\mathcal B}$, which corresponds to 
$P-\uparrow\!\!b$ for an element $b\in B$ in $P$. Then there is $a\in A$ so 
that $a\not\ge b$, so $a\not\in\uparrow\!\!b$, so ${\bf a}\le\overline{\bf b}$ 
in ${\mathcal D}$. 
Thus  ${\mathcal A}\le\overline{\mathcal B}$ in ${\mathcal D}$.

Conversely, assume that ${\mathcal A}\le\overline{\mathcal B}$ in 
${\mathcal D}$. In $P$, let $a\in A$; then ${\bf a}\in{\mathcal A}$ in 
${\mathcal D}$, and ${\bf a}\le\overline{\bf b}$ for some 
$\overline{\bf b}\in\overline{\mathcal B}$, so $a\not\ge b$ in $P$. Similarly, 
let $b\in B$; then $\overline{\bf b}\in\overline{\mathcal B}$ and 
${\bf a}\le\overline{\bf b}$ for some ${\bf a}\in{\mathcal A}$, 
which means $a\not\ge b$ in $P$. Thus $A\not\ge B$ in $P$.
\end{proof}

{\bf Note 1:} In (i), the condition that $A$ have no minimum element cannot be 
dropped. For example, let $P=\{x,y,z\}$ where $x<y$, $x<z$ and $y,z$ are incomparable, and let 
$A=B=P$. Then $A\le B$, but $A\not\ge B$ fails since $x\in B$ but there is no 
$a\in A$ so that $a\not\ge x$. A similar example shows that $B$ must have no 
maximum element.\\[-.2cm]

{\bf Note 2:}  By (ii), $A \ngeq B$ is equivalent to  $[{\mathcal A},\overline{\mathcal B}]$ being
the convex hull of $\mathcal{A} \cup \overline{\mathcal{B}}$.

\begin{proposition}\label{P:duality}
Let $A,B$ be antichains of $P$ with $|A|\le|B|$ 
and so that $A\not\ge B$ in $P$. The following are equivalent:

(i) There exists an onto function $f:B\to A$ so that, for any order ideal $X$ of 
$P$ satisfying $A\cap X\ne\emptyset$ and $B\not\subseteq X$, 
there is some $b\in B-X$ with $f(b)\in X$.

(ii) In ${\mathcal D} = \mathcal{O}(P)$, $[{\mathcal A},\overline{\mathcal B}]$ is the union of 
$|B|$ intervals $[{\bf a},\overline{\bf b}]$ with ${\bf a}\in{\mathcal A}$ and 
$\overline{\bf b}\in\overline{\mathcal B}$.
\end{proposition}

\begin{proof} 
Since $A \ngeq B$, {\bf Note 2} shows that $[{\mathcal A},\overline{\mathcal B}]$ is the
convex hull of ${\mathcal A} \cup \overline{\mathcal B}$, setting the stage for (ii) to hold.

Suppose that (i) holds.  Were $|B| = 1$, (ii) would follow immediately, so assume
$B=\{b_1,b_2,\ldots,b_n\}$ where 
$n\ge 2$. For each $i$, let 
$a_i=f(b_i)\in A$ and $X_i=P-\uparrow\!\!b_i$. Then $X_i$ is a order ideal of $P$ 
(actually, $X_i=\overline{\bf b}_i$ in ${\mathcal D}$). Since $B$ is an 
antichain, $B-X_i=\{b_i\}$ and, since $A\not\ge B$, there is some $a\in A$ so 
that $a\not\ge b_i$ and thus $a\in X_i$. By (i), $a_i=f(b_i)\in X_i$, so 
${\bf a}_i\le  X_i=\overline{\bf b}_i$ in ${\mathcal D}$, and hence 
$[{\bf a}_i,\overline{\bf b}_i]$ is an interval in ${\mathcal D}$ for all $i$. 

We claim that $[{\mathcal A},\overline{\mathcal B}]=\bigcup_{i=1}^n[{\bf a}_i,
\overline{\bf b}_i]$. Let $X \in[{\mathcal A},\overline{\mathcal B}]$,
which means that 
${\bf a}\le X \le\overline{\bf b}$ for some ${\bf a}\in{\mathcal A},
\overline{\bf b}\in\overline{\mathcal B}$.  Then $a\in A\cap X$ 
and $b\in B-X$ in $P$, so by (i) there must be $b_i\in B-X$ so that 
$a_i=f(b_i)\in X$. Therefore ${\bf a}_i\le  X \le\overline{\bf b}_i$ in 
${\mathcal D}$, which proves the claim and thus (ii).

Conversely, suppose that (ii) holds, say 
$[{\mathcal A},\overline{\mathcal B}]
=\bigcup_{i=1}^n [{\bf a}_i,\overline{\bf b}_i]$. Since $A$ and $B$ are 
antichains, so are ${\mathcal A}$ and $\overline{\mathcal B}$ (in 
${\mathcal D}$), and so every element of ${\mathcal A}$ and 
$\overline{\mathcal B}$ must occur as an endpoint of at least one of these 
intervals. Thus $\overline{\mathcal B}
=\{\overline{\bf b}_1,\overline{\bf b}_2,\ldots,\overline{\bf b}_n\}$, and 
$f$ defined by $f(b_i)=a_i$ for all $i$ is a surjection of $B$ onto $A$. Let 
$X$ be a order ideal of $P$ satisfying $A\cap X\ne\emptyset$ and 
$B\not\subseteq X$. Then there exists $a\in A\cap X$ and $b\in B-X$, which means 
that ${\bf a}\le X \le\overline{\bf b}$ in ${\mathcal D}$. Thus
$X \in[{\mathcal A},\overline{\mathcal B}]$, and by (ii) there is 
${\bf a}_i\in{\mathcal A}$ and 
$\overline{\bf b}_i\in\overline{\mathcal B}$ so that 
$X \in[{\bf a}_i,\overline{\bf b}_i]$ . Then in $P$, $b_i\in B-X$ and 
$f(b_i)=a_i\in X$, so (i) holds.
\end{proof}

{\bf Note 3:} The condition that $A$ and $B$ are antichains is required. For 
example, let $P$ be the 4-element poset 

\begin{picture}(10,60)(-50,0)
\put(0,0){\line(0,1){20}}
\put(20,0){\line(0,1){20}}
\put(20,0){\line(-1,1){20}}
\put(0,0){\circle*{3}}
\put(0,20){\circle*{3}}
\put(20,0){\circle*{3}}
\put(20,20){\circle*{3}}
\put(-13,0){$a_1$}
\put(-13,20){$b_2$}
\put(25,0){$b_1$}
\put(25,20){$a_2$}
\put(50,10){for which ${\mathcal D}={\mathcal O}(P)$ is}
\multiput(240,-30)(20,20){3}{\line(-1,1){20}}
\multiput(240,10)(20,20){2}{\line(-1,1){20}}
\multiput(240,-30)(-20,20){2}{\line(1,1){40}}
\put(220,30){\line(1,1){20}}
\multiput(240,-30)(-20,20){2}{\multiput(0,0)(20,20){3}{\circle*{3}}}
\put(207,-13){$a_1$}
\put(207,27){$b_2$}
\put(265,-13){$b_1$}
\put(285,7){$a_2$}
\put(202,7){$a_1\vee b_1$}
\put(265,30){$a_1\vee a_2$}
\put(310,10){$.$}
\end{picture}

\vspace{40pt}\noindent
Let $A=\{a_1,a_2\}$ and $B=\{b_1,b_2\}$ and note that $A\not\ge B$ holds but 
that $B$ is not an antichain. Also, $X=\{a_1,b_1\}$ forces $f(b_2)=a_1$, while 
$X=\{a_2,b_1\}$ forces $f(b_2)=a_2$. So no such onto function $f:B\to A$ is 
possible, and (i) fails. However, in the corresponding 
distributive lattice ${\mathcal D}$, we get that ${\mathcal A}=\{a_1,a_2\}$ 
and $\overline{\mathcal B}=\{a_1,a_1\vee a_2\}$, and 
$$[{\mathcal A},\overline{\mathcal B}]=\{a_1,a_2,a_1\vee b_1,a_1\vee a_2\}
=[a_1,a_1\vee a_2]\cup[a_2,a_1\vee a_2]$$ 
is the union of $|B|=2$ intervals. Thus (ii) holds. Switching the roles of $A$ 
and $B$ gives an example showing $A$ must be an antichain.\\[-.2cm]

{\bf Note 4:} Unfortunately, conditions (i) and (ii) of Proposition \ref{P:duality} need not hold for 
all antichains $A,B$ satisfying the given conditions. For example, let $P$ be 
the 4-element poset 

\begin{picture}(10,60)(-50,0)
\put(0,0){\line(0,1){20}}
\put(20,0){\line(0,1){20}}
\put(0,0){\circle*{3}}
\put(0,20){\circle*{3}}
\put(20,0){\circle*{3}}
\put(20,20){\circle*{3}}
\put(-13,0){$a_1$}
\put(-13,20){$b_2$}
\put(25,0){$b_1$}
\put(25,20){$a_2$}
\put(50,10){for which ${\mathcal D}={\mathcal O}(P)$ is}
\multiput(240,-30)(20,20){3}{\line(-1,1){40}}
\multiput(240,-30)(-20,20){3}{\line(1,1){40}}
\multiput(240,-30)(-20,20){3}{\multiput(0,0)(20,20){3}{\circle*{3}}}
\put(207,-13){$a_1$}
\put(187,7){$b_2$}
\put(265,-13){$b_1$}
\put(285,7){$a_2$}
\put(211,14){$a_1\vee b_1$}
\put(265,30){$a_1\vee a_2$}
\put(310,10){$.$}
\end{picture}

\vspace{40pt}\noindent
Let $A=\{a_1,a_2\}$ and $B=\{b_1,b_2\}$ and note that $A\not\ge B$ holds. 
However, $X=\{a_1,b_1\}$ forces $f(b_2)=a_1$, while $X=\{a_1,b_2\}$ forces 
$f(b_1)=a_1$. So no 
such onto function $f:B\to A$ is possible. Looking at the corresponding 
distributive lattice ${\mathcal D}$, we get that ${\mathcal A}=\{a_1,a_2\}$ 
and $\overline{\mathcal B}=\{b_2,a_1\vee a_2\}$, and 
$[{\mathcal A},\overline{\mathcal B}]=\{a_1,a_2,a_1\vee b_1,a_1\vee a_2,b_2\}$ 
is not the union of two intervals. Thus both (i) and (ii) fail.

\section{Proofs of Theorems \ref{T:distributive} and \ref{T:poset}}\label{S:Proof} 

{\it Proof of Theorem 2.}  Recall that $A,B$ are antichains in $P$ with 
$|A|\le|B|$ and $A\le B$ in $P$. We aim to prove there exists an onto 
function $f:B\to A$ so that, for any order ideal $X$ of 
$P$ satisfying $A\cap X\ne\emptyset$ and $B\not\subseteq X$, 
there is some $b\in B-X$ with $f(b)\in X$.

 First note that the result is trivially true if 
$|A|=1$, so we assume that $|A|\ge 2$. 
To obtain a contradiction, suppose that $P$ is a partially ordered set of 
minimum cardinality for which the theorem is false and, among these 
counterexamples, take $P$ to have a minimum number of comparabilities. We 
now constrain the structure of $P$ to obtain a contradiction. \\[-.2cm]

{\bf Claim 1:}  $P=A\cup B$. 

Otherwise, suppose that $A\cup B\subset P$, 
so that the theorem is true for the proper subposet $P'=A\cup B$.
Let $f:B\to A$ be the function given by the theorem in that 
case. Looking now at $P$, let $X$ be an order ideal of $P$ satisfying 
$A\cap X\ne\emptyset$ and $B\not\subseteq X$. Then $X'=X\cap P'$ is a order ideal 
of $P'$ satisfying $A\cap X'\ne\emptyset$ and $B\not\subseteq X'$, so by the 
theorem applied to $P'$ there is some $b\in B-X'$ with $f(b)\in X'$. Since 
$B-X'=B-X$ and $X'\subseteq X$, the theorem is true for $P$, a contradiction. 
Thus $P=A\cup B$, verifying the claim.

Note that, with $P=A\cup B$ and the conditions that $A$ and $B$ are antichains 
satisfying $A\le B$, we have that $P$ is bipartite with minimals $A$ and 
maximals $B$ ($A\cap B$ being the isolated points if any). \\[-.2cm]

{\bf Claim 2:}  $P$ has no isolated elements.

To see this, suppose that $P=Q\cup\{c\}$ where $c$ is incomparable to all elements of $Q$. 
Note that $c\in A\cap B$.  Let $A'=A-\{c\}$ and $B'=B-\{c\}$. Since $|B|\ge|A|\ge 2$, $A'$ and $B'$ are nonempty; also,
 $A'\le B'$ and $|A'|=|A|-1\le|B|-1=|B'|$. Thus the hypotheses of the 
theorem hold for $Q,A'$ and $B'$.

Let $f:B' \to A'$ be the map guaranteed by the minimality of $P$.
Choose any $a \in A'$ and let $f(b_0) = a$ for some $b_0 \in B'$.  Define $g:B \to A$ by
 \begin{equation*}
     g(b) =     
     \begin{cases}
       f(b)    & \text{if} \ b \ne b_0, c\\
       c        & \text{if} \ b = b_0\\
       a        & \text{if} \ b = c,
     \end{cases}
   \end{equation*}
and observe that $g$ maps $B$ onto $A$.

Let $Y$ be an order ideal of $P$ satisfying $Y\cap A\ne\emptyset$ and $B\not\subseteq Y$. The following exhaust the possibilities for $Y$.

\begin{itemize}
\item If $c\in Y$ and $b_0\not\in Y$ then $g(b_0)=c\in Y$.\\

\item If $c\in Y$ and $b_0\in Y$, let $X=Y-\{c\}$. Then $X$ is an order ideal of $Q$, $X\cap A'\ne\emptyset$ (since $b_0\in Y$ implies $Y\cap A'\ne\emptyset$) and $B'\not\subseteq X$, so by assumption there is some $b\in B'-X=B-Y$ so that $f(b)\in X$. Since $b\ne b_0$, $g(b)=f(b)\in X\subseteq Y$.\\

\item If $c\not\in Y$ then $Y$ is an order ideal of $Q$ and $Y\cap A'\ne\emptyset$. If $B'\subseteq Y$, then $Y=Q$, and $g(c)=a\in Y$. If $B'\not\subseteq Y$, then by assumption there is some $b\in B'-Y\subseteq B-Y$ so that $f(b)\in Y$. If $b\ne b_0$ then $g(b)=f(b)\in Y$. If $b=b_0$ then $f(b)=a\in Y$, so $g(c)=a\in Y$ while $c\in B-Y$.
\end{itemize} 

Thus $g$ has the required properties for $P,A$ and $B$, a contradiction. {\bf Claim 2} follows.\\[-.2cm]

A {\it star} is a partially ordered set of length 1 whose covering graph (or comparability graph) is a star in the usual graph-theoretic
sense.\\[-.2cm]

{\bf Claim 3:} $P$ is a length one partially ordered set each of whose connected components is a star (allowing 2-element stars
but no 1-element components).

If some connected component of $P$ is not a star then it must contain 4 elements, say $a, c \in A$ and $b, d \in B$, with the covering 
relations $a \prec b, b \succ c, c \prec d$.  Let $Q$ be obtained from $P$ by
removing the single relation $b \succ c$.  Then in $Q$, $A$ is still the set of minimals and $B$, the set of maximals. Since $Q$ has fewer comparabilities than $P$, by assumption there is $f:B\to A$ satisfying the condition in the theorem. Let $X$ be an order ideal of $P$ so that $A\cap X\ne\emptyset$ and $B\not\subseteq X$. Then $X$ is still an order ideal in $Q$, so there is some $b\in B-X$ with $f(b)\in X$. These conditions still hold in $P$, showing that the same function $f$ enables $P$ to satisfy the theorem, which is a contradiction. This verifies {\bf Claim 3}.
 
 We are now ready to complete the proof by defining an onto function $f:B \to A$ with the required properties. 
 
\begin{figure}[h]
   \scalebox{.4}{\includegraphics{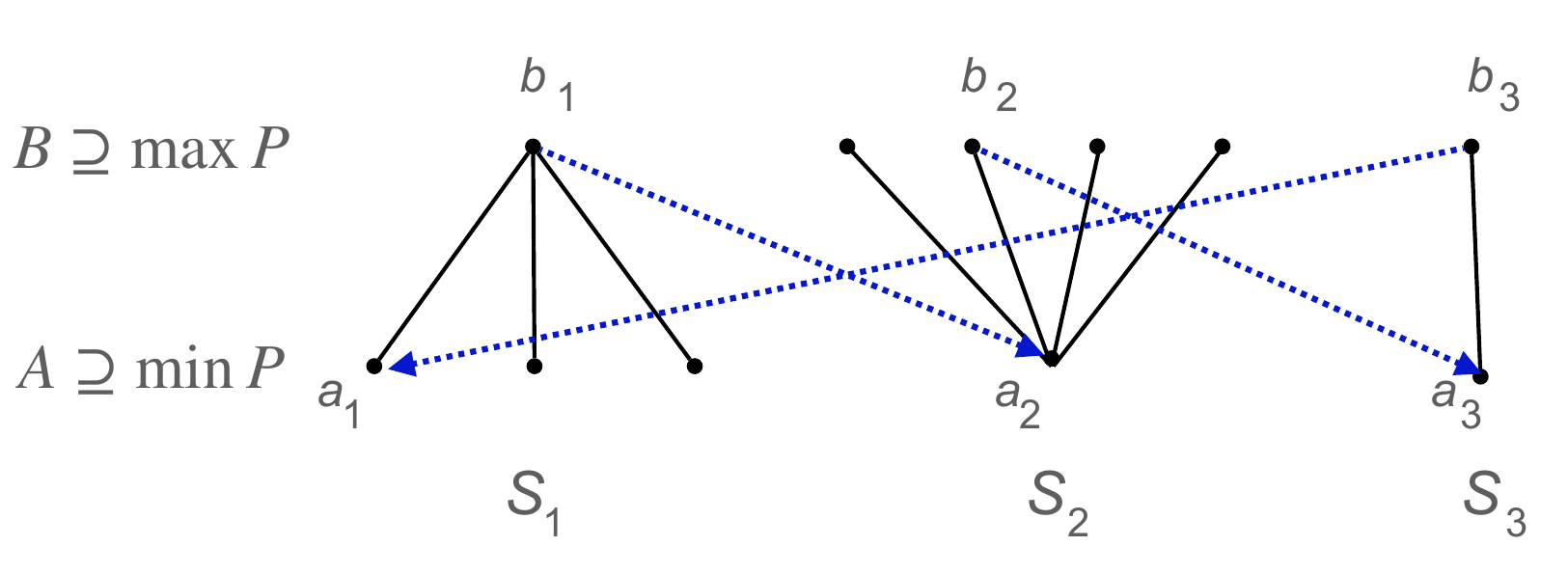}}
   \caption{Complete the surjection $f:B \mapsto A$ arbitrarily}\label{F:stars}
\end{figure}


Suppose that $P$ is the union of disjoint stars $S_i,\  i = 1, 2, \ldots, n$, each oriented as a bipartite
partially ordered set.  For each $i$, let $a_i$ be a minimal element and $b_i$ be a maximal element
of $S_i$.   Define $f:B \to A$ as follows: $f(b_i) = a_{i+1}$, modulo $n$, and complete the definition
of $f$ as a surjection arbitrarily.  (See Figure \ref{F:stars}.)  Note that in case $n = 1$, we have $f(b_1) = a_1$ and $f$, as defined,
satisfies the conclusion of the theorem.

Let $B' = \{b_1, b_2, \ldots, b_n\}$. Suppose that $X$ is a proper order ideal of $P$, that is, $A\cap X\ne\emptyset$ 
and $B\not\subseteq X$.
  \begin{itemize}
      \item If $X \cap B' \ne \emptyset$ and $B' \nsubseteq X$ then choose $i$ such that $b_{i-1} \notin X$
      and $b_i \in X$.  Then $f(b_{i-1}) = a_i \in \ X$ since each $a_j < b_j$.\\
      \item If $X\cap B' = \emptyset$ then pick some $a_i \in  X$ and see that $f(b_{i-1}) = a_i$,
      where $b_{i - 1} \in B - X$.\\
      \item If $B' \subseteq X$ then for any choice of $b \in B - X$, $f(b) \in X$, since either
      $f(b) = a_i \in X$ or $f(b) = a$, distinct from all the $a_j$'s, so $a < b_i \in X$ for some $i$. 
  \end{itemize}
This completes the proof of Theorem \ref{T:poset}. $\square$\\[-.2cm]

\begin{theorem}\label{T:general}
 Let ${\mathcal D}$ be a distributive lattice, and let 
${\mathcal A}$ and 
${\mathcal B}$ be antichains of join-irreducible elements of ${\mathcal D}$ 
satisfying $2 \le |{\mathcal A}|\le|{\mathcal B}|$ and 
${\mathcal A}\le{\mathcal B}$. Then 
$[{\mathcal A},\overline{\mathcal B}]$ is the union of $|{\mathcal B}|$ 
intervals $[{\bf a},\overline{\bf b}]$ with ${\bf a}\in{\mathcal A}$ and 
$\overline{\bf b}\in\overline{\mathcal B}$.
\end{theorem}

\begin{proof}
With $A$ and $B$ being the counterparts in $P$ of 
${\mathcal A}$ and ${\mathcal B}$ respectively, it follows from 
${\mathcal A}\le{\mathcal B}$ that $A\le B$. Since $|A|\ge 2$ and $A$ and $B$ 
are antichains, Lemma \ref{L:orders2} implies that $A\not\ge B$ and thus 
${\mathcal A}\le\overline{\mathcal B}$. Therefore the 
convex subset $[{\mathcal A},\overline{\mathcal B}]$ in ${\mathcal D}$ is 
defined, and Theorem~\ref{T:poset} and Proposition \ref{P:duality} give the result.
\end{proof}

{\bf Note 5:}  The condition that  $2 \le |{\mathcal A}|$ is needed.  If $\mathcal{A}$ is
a singleton and  $[{\mathcal A},\overline{\mathcal B}]$  is an interval, the result is
trivial.  However, $[{\mathcal A},\overline{\mathcal B}]$ need not even be convex if
$|\mathcal{A}| = 1$.  For instance, take $A = B = \{a\}$.  Then  ${\mathcal A} \le \overline{\mathcal B}$
does not hold.
\vspace{.2cm}

\begin{corollary}
{\bf (Theorem \ref{T:distributive})}. Let ${\mathcal D}$ be a distributive lattice with  $|{\mathcal D}| \ge 3$,  
atoms ${\mathcal A}$ and coatoms ${\mathcal C}$. Then ${\mathcal D}-\{0,1\}$ 
is the union of $\max (|{\mathcal A}|,|{\mathcal C}|)$ intervals.
\end{corollary}

\begin{proof}
 The result is trivial if either $\mathcal{A}$
 or $\mathcal{C}$ is a singleton.  So, by symmetry, we may assume that 
$2 \le |{\mathcal A}|\le|{\mathcal C}|$. Note that ${\mathcal A}$ is the set of minimal 
join-irreducibles of ${\mathcal D}$, while 
${\mathcal C}=\overline{\mathcal B}$ where 
${\mathcal B}$ is the set of maximal join-irreducibles of ${\mathcal D}$. 
Also $|{\mathcal B}|=|\overline{\mathcal B}|$ and 
${\mathcal A}\le{\mathcal B}$, so the result follows from Theorem \ref{T:general}.
\end{proof}

\section{Proofs of Theorems \ref{T:size2} and \ref{T:slight}}\label{S:more-dist} 

Our interest now is mainly focussed on the convex subset $[A,B]$ between levels $A$ and $B$ of a distributive lattice $D$ 
and on two questions.\\[-.2cm]

 {\bf Question 1:}  When is $[A,B]$ minimally coverable?\\[-.2cm]
  
 {\bf Question 2:}  How many intervals are required to cover $[A,B]$? \\[-.2cm]
 
Theorems  \ref{T:size2}  and \ref{T:slight}, stated in \S\ref{S:intro}, provide limited answers to Questions 1 and 2 respectively. After proving these two theorems, we will speculate on possible improvements to these questions and answers. 

Having obtained Theorem \ref{T:distributive}  from Theorem \ref{T:general}, we must now look to other techniques, not involving the poset $J(D)$ of 
join irreducibles of $D$. Thus in this section we abandon the calligraphic and boldface notation for lattices and their elements 
used previously.
For purposes of the proofs of Theorems \ref{T:size2} and \ref{T:slight}, we adopt the following notation. 
For elements $x < y$ in a partially ordered set $P$, 
 we let $\overline{[x,y]}$ denote any maximal chain of $P$ containing both $x$ and $y$.   Also, we
denote that $x$ and $y$ are an incomparable pair by $x \inc y$.

{\bf Theorem \ref{T:size2}.} {\it
   Let $D$ be a distributive lattice with distinct levels $A$ and $B$ that satisfy $A \le B$ and $\min(|A|, |B|) = 2$.  
   Then $[A, B]$ has a cover by $\max(|A|, |B|)$ intervals.}

 \begin{proof}
 By dualizing $D$ if necessary, we may assume $A = \{a, b\}$.  First, let's see that $a \wedge b \prec a$.  Otherwise
   there is some $x \in D$ with $a \wedge b < x < a$.  Since $\{a, b\}$ is a level
   of $D$ and $D$ is ranked, there is some $y \ne x$ in $D$ such that $a \wedge b < y < b$.
   Then $x \vee y \nleq a$ since $a \wedge b < y$; and, $a \nless x \vee y$ by
   distributivity. Thus $a||x\vee y$ and similarly $b||x\vee y$, and so $\{a, b, x\vee y\}$ is a 3-element antichain, a contradiction. 
   Of course, $a \wedge b \prec b$ as well and, again because there is no 3-element antichain containing $a, b$, the only covers of $a \wedge b$ in $D$ are $a$ and $b$. Thus in the distributive (sub)-lattice $[a\wedge b,1_D]$, $\{a,b\}$ is the set of atoms.   
   
   It follows that the convex set $[A, B]$ in $D$
   is unchanged in $[a \wedge b, 1_D]$.  So, we may assume that $A$ is the
   set of atoms of $D$.  With $B =\{b_1, b_2, \ldots, b_n\}$ we shall prove that
   $[A, B] = [x_1, b_1]\cup[x_2, b_2]\cup\cdots\cup[x_n, b_n]$ for some $x_1, x_2, \ldots,x_n\in\{a,b\}$.
   
   Each of $a$ and $b$ must be less than at least one 
element of $B$, and likewise each element of $B$ must be greater than $a$ or 
$b$ (or both). Without loss of generality we can assume that $a<b_i$ for all 
$i\in\{1,2,\ldots,k\}$ and that $a\;||\;b_j$ (and thus 
$b<b_j$) for all $j\in\{k+1,\ldots,n\}$, where $k$ is some integer in 
$\{1,2,\ldots,n\}$. 

\medskip
First assume that $k<n$. Then we claim that $x_i=\left\{\begin{array}{ll}
a&\mbox{if } i\le k\\b&\mbox{if } i>k
\end{array}\right.$ works, that is,
$$[A, B] = [a,b_1]\cup\cdots\cup[a,b_k]\cup[b,b_{k+1}]\cup\cdots\cup
[b,b_n].$$ 
For otherwise there exists some $c\in [A, B]$ not in this union. By definition of 
$[A, B]$, we may assume that $c\in[b,b_1]$, so $b\le c\le b_1$. Since $c$ is 
not in the above union, we get $c\;||\;a$ and $c\;||\;b_j$ for all 
$j\in\{k+1,\ldots,n\}$. Since $a$ is an atom of $D$ and $k<n$, 
$a\wedge c=a\wedge b_n=0$. Also, the maximal chain 
$\overline{[c,c\vee b_n]}$ cannot contain any of $b_{k+1},\ldots b_n$, 
but must intersect $B$, 
and so $b_i\in\overline{[c,c\vee b_n]}$ for some $i\le k$. Thus 
$b_i<c\vee b_n$. But now by distributivity 
$$0=(a\wedge c)\vee(a\wedge b_n)=a\wedge(c\vee b_n)\ge a\wedge b_i=a,$$ 
a contradiction.

\medskip
Thus $k=n$, which means that $a<b_i$ for all $i$. By symmetry we may also 
assume that $b<b_i$ for all $i$, so in other words the subposet of $D$ 
formed from $A\cup B$ is complete bipartite. In this case we claim that the 
choice  $x_i=\left\{\begin{array}{ll}
a&\mbox{if } i\ne k\\b&\mbox{if } i=k
\end{array}\right.$ works for some $k\in\{1,\ldots,n\}$; that is, at least 
one of the $n$ families of intervals
\begin{eqnarray*}
{\mathcal F}_1 &=& \{[b,b_1],[a,b_2],[a,b_3]\ldots,[a,b_n]\}\\
{\mathcal F}_2 &=& \{[a,b_1],[b,b_2],[a,b_3]\ldots,[a,b_n]\}\\
&\vdots&\\
{\mathcal F}_n &=& \{[a,b_1],[a,b_2],\ldots,[a,b_{n-1}],[b,b_n]\}
\end{eqnarray*}
has union equal to all of $[A, B]$. 

\medskip
Suppose not. Then there exist $c_1,\ldots,c_n\in [A, B]$ so that, for all $i$, 
$c_i\not\in\bigcup{\mathcal F}_i$. 

\medskip
Suppose first that $c_1\in[a,b_1]$ and $c_2\in[a,b_2]$. Then 
$a\le c_1\le b_1$, so 
$c_1\not\in\bigcup{\mathcal F}_1$ implies that $c_1$ is incomparable to 
$b_2,b_3,\ldots,b_n$ and $b$. Similarly we know that $c_2$ is 
incomparable to $b_1,b_3,\ldots,b_n$ and $b$. Thus 
$c_1\vee c_2\not\le b_i$ for any $i\in\{1,\ldots, n\}$. Also, since $b$ is 
an atom, $c_1\wedge b=c_2\wedge b=0$, so by distributivity 
$(c_1\vee c_2)\wedge b=0$, so $c_1\vee c_2\;||\; b$. Thus 
$c_1\vee c_2\not\ge b_1,b_2,\ldots,b_n$, and so $c_1\vee c_2$ is 
incomparable 
to the entire level $B$, which is impossible. We conclude by symmetry that 
$c_i\in[a, b_i]$ can occur for at most one value of $i$.

\medskip
Next, suppose that in fact $c_i\not\in[a,b_i]$ for any $i$. Since 
$c_i\in [A, B]-\bigcup{\mathcal F}_i$, this forces that for all $i$ there is 
$j\ne i$ so that $c_i\in[b,b_j]$, which 
means $b\le c_i\le b_j$. Since $c_i\not\in\bigcup{\mathcal F}_i$, 
$c_i\not\in[b,b_i]\cup[a,b_j]$, which implies that $c_i\;||\;b_i$ 
and $c_i\;||\;a$. Thus 
$c_i\wedge a=0$ for all $i$, which by distributivity means that 
$(c_1\vee\cdots\vee c_n)\wedge a=0$ and so $c_1\vee\cdots\vee c_n\;||\; a$. It 
follows that $c_1\vee\cdots\vee c_n\not\ge b_i$ for any $i$, while 
$c_i\;||\;b_i$ means that $c_1\vee\cdots\vee c_n\not\le b_i$ for any $i$. 
Therefore $c_1\vee\cdots\vee c_n\;||\;b_i$ for all $i$, a contradiction.

\medskip
So we may now assume by symmetry that $c_1\in[a,b_1]$ and that, for all 
$i>1$, $c_i\in[b,b_j]$ for some $j\ne i$. It follows that 
$b\le c_2,c_3,\ldots,c_n$, and that for each $i>1$ there is $j\ne i$ so that 
$c_i\le b_j$. Since  $c_i\not\in\bigcup{\mathcal F}_i$ and  
$[a,b_j]\in{\mathcal F}_i$, we get that $c_i\;||\;a$ for all $i>1$. By 
distributivity, $c_2\vee\cdots\vee c_n\;||\; a$. Thus 
$c_2\vee\cdots\vee c_n\not\ge b_i$ for any $i$. Since $c_2,\ldots,c_n\ge b$ 
while  $c_i\not\in\bigcup{\mathcal F}_i$ and $[b,b_i]\in{\mathcal F}_i$, 
we also know that $c_2\vee\cdots\vee c_n\not\le b_i$ for any $i>1$. Thus, 
since the maximal chain 
$\overline{[c_2\vee\cdots\vee c_n\;,\;c_2\vee\cdots\vee c_n\vee b_2]}$ 
avoids $b_2,\ldots,b_n$ and so must contain $b_1$, we have that 
\begin{equation}
b_1<c_2\vee\cdots\vee c_n\vee b_2.
\end{equation} 
From $c_1\in[a,b_1]$ we know that $c_1\le b_1$, and so 
$c_1\not\in\bigcup{\mathcal F}_1$ means that $c_1\;||\; b$. Thus 
$c_1\wedge(c_2\vee\cdots\vee c_n)\not\ge b$, and also 
$c_1\wedge(c_2\vee\cdots\vee c_n)\not\ge a$ since 
$c_2\vee\cdots\vee c_n\;||\;a$; therefore  
\begin{equation}
c_1\wedge(c_2\vee\cdots\vee c_n)=0.
\end{equation}
Now we look at the maximal chain $\overline{[c_1,c_1\vee b_2]}$. It cannot 
contain any of $b_2,b_3,\ldots, b_n$, because $a\le c_1$ and 
$[a,b_2],\ldots,[a,b_n]\in{\mathcal F}_1$ while 
$c_1\not\in\bigcup{\mathcal F}_1$. So $\overline{[c_1,c_1\vee b_2]}$ must 
contain $b_1$, that is, $b_1<c_1\vee b_2$. 
But now, by equations (1), (2) and distributivity,
$$b_1\le(c_1\vee b_2)\wedge(c_2\vee\cdots\vee c_n\vee b_2)
=(c_1\wedge(c_2\vee\cdots\vee c_n))\vee b_2=0\vee b_2=b_2,$$
a contradiction.
\end{proof}

The proof of Theorem \ref{T:size2}  starts by showing that we can assume that the two-element level is the atoms. The fact that we apparently need to reduce Theorem  \ref{T:size2} to this situation, and the proof then goes through, could be construed as some support for the reasonableness of the following problem.

\begin{problem}\label{Prob:atoms}  Determine whether $[A,B]$ is minimally coverable whenever $A$ is the set of atoms of
a distributive lattice $D$ and $B$ is another level of $D$.\end{problem}

Besides Theorem  \ref{T:size2}, as support for Problem \ref{Prob:atoms} we can only report a ridiculously long affirmation of Problem 2 in the case $|A|=|B|=3$. This proof will not be offered here but can be obtained from the authors by anyone interested in seeing it.  On the other side, the example in Figure \ref{F:Glued} shows that 
the same result for levels that do not include the atoms or coatoms is unlikely.

Turning now to Question 2 posed at the beginning of this section, Theorem  \ref{T:slight} from \S\ref{S:intro} provides a slight improvement on the obvious upper bound of $|A| \cdot |B|$.  Our proof requires a technical fact about distributive lattice levels.

\begin{lemma}\label{l:unique}
Let $D$ be a distributive lattice and let $L$ and $U$ be levels 
of $D$ with $a \in L$ and $b, c\in U$, $b \ne c$.  If there exist $x,y\in D$ so that $a<x<b$, $a<y<c$,  
$x \inc v$ for all $v \in (L\cup U)-\{a, b\}$, and $y \inc w$ for all 
$w \in (L\cup U)-\{a,c\}$, then $L=\{a\}$.
\end{lemma}

\begin{proof}
The maximal chain $\overline{[x, x \vee y]}$ must intersect $U$, 
and since $x$ is comparable only to $b$ in $U$,  $\overline{[x, x \vee y]}$ 
must contain $b$. Since $y \inc b$, we cannot 
have $x \vee y \le b$, so $x\vee y >b$. By symmetry, $x\vee y >c$. 

\noindent
Now suppose that $L\ne\{a\}$, and let $d\in L$, $d\ne a$. The maximal chain 
$\overline{[x \wedge d ,x ]}$ must intersect $L$, and since $x$ is comparable only 
to $a$ in $L$, we know that $x\wedge d < a$. Similarly, $y\wedge d<a$. Thus
$$x\wedge d=a \wedge d=y \wedge d=(x\vee y)\wedge d \ge b \wedge d,$$
so $x > b \wedge d$. But now
$$(x \vee d)\wedge b =( x \wedge b) \vee (d\wedge b)=x\vee(b \wedge d)=x,$$
and so $x\vee d$ is incomparable to $b$. But this means that the maximal chain 
$\overline{[x, x \vee d]}$ does not contain $b$. Since $x$ is only comparable 
to $b$ in $U$, $\overline{[x, x \vee d]}$ cannot 
intersect $U$ at all, a contradiction. \end{proof}

{\bf Theorem \ref{T:slight}.} {\it 
  Let $D$ be a distributive lattice with distinct levels $A$ and $B$ that satisfy $A \le B$ and $|A|, |B| > 1$. 
  Then $[A, B]$ has a cover by at most

\begin{enumerate}

  \item $|A| \cdot |B| - \min (|A|, |B|)$ intervals in $D$, if $|A| \ne |B|$ or $|A| = |B|$ is even,\\
  
  \item $|A| \cdot |B| -  \min (|A|, |B|)+ 1$ intervals in $D$, if  $|A| = |B|$ is odd.
  
\end{enumerate}}

\begin{proof}
Let $A=\{ a_1,\ldots, a_m\}$ and $B=\{b_1,\ldots,b_n\}$. Then 
$$[A, B]=\cup\{[a_i, b_j]\;|\;1\le i\le m,1\le j\le n\}$$
where of course $[a_i, b_j] = \emptyset$ whenever $a_i \not< b_j$. By dualizing $D$ if necessary, we can assume
that $m \le n$.

Let $s$ and $t$ be positive integers satisfying $s\le m$ and $s\le t\le n$. Consider the $t$ sets of intervals 
  \begin{align*}
        {\mathcal S}_0 &= \{[a_i, b_j]\;|\;1\le i\le m,1\le j\le n\}-\{[a_1,b_1], [a_2,b_2],\ldots,[a_s, b_s]\},\\
        {\mathcal S}_1 &= \{[a_i,b_j]\;|\;1\le i\le m,1\le j\le n\}-\{[a_1,b_2], [a_2,b_3],\ldots,[a_s,b_{s+1}]\},\\[-.2cm] 
                                 & \vdots\\[-.2cm]
        {\mathcal S}_{t-1} &= \{[a_i,b_j]\;|\;1\le i\le m,1\le j\le n\}-\{[a_1,b_t], [a_2,b_{t+1}],\ldots,[a_s,b_{t+s-1}]\}.
   \end{align*}
Here, the subscripts of the $b$'s in the intervals $[a_1,b_{k+1}], \ldots, [a_s,b_{k+s}]$
missing from each $S_k$ are taken modulo $t$. Thus the second coordinates of the intervals missing from each $S_k$
form an $s$-element subset of $\{b_1, \ldots, b_t\}$. Each 
${\mathcal S}_i$ has at most $mn-s$ members. Note that, since $s\le m$ and $t\le n$, 
the (at most) $st$ intervals missing
from the various ${\mathcal S}_i$'s are all distinct.

\medskip
Suppose that none of the ${\mathcal S}_i$'s covers all of $[A, B]$. In 
particular ${\mathcal S}_0$ does not 
cover all of $[A, B]$, so there is some $z_0\in [A, B]$ not in 
any interval of ${\mathcal S}_0$, which means that 
$z_0\in\cup_{i=1}^s[a_i,b_i]$. Without loss of generality, 
$z_0\in[a_1,b_1]$. Suppose that also $z_0\in[a_j,b_j]$ 
for some $j>1$; then $a_1 \le z_0 \le b_1$ and $a_j\le z_0 \le  b_j$, so 
$z_0\in[a_1,b_j]\in{\mathcal S}_0$, which contradicts the choice of 
$z_0$. Thus, among all 
the intervals $[a_i,b_j]$,  $z_0$ lies only in the interval 
$[a_1,b_1]$. Similarly, for each $k\in\{0,1,\ldots,t-1\}$ there is an 
element $z_k\in [A, B]$ which lies in a unique interval 
$[a_{i_k},b_{i_k+k}]$ among all intervals $[a_i,b_j]$.   Since $z_k \notin A \cup B$, $a_{i_k} < z_k < b_{i_{k}+k}$
for all $k$.

\medskip
{\bf Case 1}: Assume that $m < n$.  Let $s=m$, $t=m+1$. Since there are $m+1$ 
$z_k$'s and corresponding intervals, and only $m$ elements of $A$, there must 
be two 
$z_k$'s whose intervals have the same left endpoint $a_{i_k}$ (but different 
right endpoints, since $n\ge t$). Since $m>1$, this is  a contradiction to
Lemma \ref{l:unique}. Therefore one of the ${\mathcal S}_i$'s must cover all of 
$[A, B]$. Since each ${\mathcal S}_i$ has at most $mn-m$ intervals, this proves 
Theorem \ref{T:slight} in the case $m \ne n$.

\medskip
{\bf Case 2}:  Assume that $m=n$. Let 
$$s=t=\left\{\begin{tabular}{ll}$m$ & if $m$ is even,\\$m-1$ & if $m$ is odd.
\end{tabular}\right.$$ 
(In other words, $s=t$ is the largest even integer at most $m=n$.) 
Note that $s\ge 2$ since $m>1$. If any two of the $i_k$'s coincide, then 
(since $s>1$) we would have a contradiction to Lemma \ref{l:unique}. 
We argue similarly, using the dual of Lemma \ref{l:unique},  if any of the $i_k+k$'s coincide. So we must have the following 
situation: for each $k$, $z_k$ lies in a unique interval $[a_{f(k)},b_{g(k)}]$, 
where $\{f(0),\ldots,f(s-1)\}=\{1,2,\ldots,s\}=\{g(0),\ldots,g(s-1)\}$  and 
$g(k)\equiv f(k)+k\pmod s$ for each $k$. Thus 
$$\sum_{k=0}^{s-1} f(k)=\frac{s(s+1)}{2}=\sum_{k=0}^{s-1} g(k)\equiv
\sum_{k=0}^{s-1} f(k)+\sum_{k=0}^{s-1} k\pmod s,$$
so 
$$0\equiv\sum_{k=0}^{s-1} k=\frac{s(s-1)}{2}\pmod s,$$
which happens exactly if $s$ is odd. Since $s$ is even, we get a 
contradiction. Thus one of the ${\mathcal S}_i$'s must cover $[A, B]$,  
and Theorem \ref{T:slight} follows.
\end{proof}

Regarding Question 2 posed at the beginning of this section, Theorem \ref{T:slight} shows 
that the answer to this question is wide open.  From the example in Figure \ref{F:Glued}, 
we optimistically propose the following:

\begin{problem}\label{Prob:upper}  Given levels $A \le B$ in a distributive lattice $D$, determine if $[A, B]$ can be covered by $|A|+|B|-2$ intervals.  If not, can $|A| + |B| + 2$ be replaced by some linear function of $\max(|A|,|B|)$? \end{problem}



\end{document}